\documentclass{amsart}[12pt]
\usepackage{amsmath,amsfonts,amssymb,latexsym,amsthm}
\usepackage{epsfig}
\usepackage{hyperref} %hyperref wants to be loaded last
\usepackage[dvipsnames]{xcolor}

\usepackage{ytableau}

\numberwithin{equation}{section}

\newtheorem{thm}{Theorem}[section]
\newtheorem{lemma}[thm]{Lemma}

\newtheorem{cor}[thm]{Corollary}

\newtheorem{conj}[thm]{Conjecture}

\newtheorem{rem}[thm]{Remark}

\usepackage[colorinlistoftodos]{todonotes}

\RequirePackage{cleveref}
\usepackage{hypcap}
\hypersetup{colorlinks=true, citecolor=darkblue, linkcolor=darkblue}
\definecolor{darkblue}{rgb}{0.0,0,0.7}
\newcommand{\darkblue}{\color{darkblue}}

\definecolor{darkred}{rgb}{0.68,0,0}
\newcommand{\darkred}{\color{darkred}}

\definecolor{darkgreen}{rgb}{0,.38,0}
\newcommand{\darkgreen}{\color{darkgreen}}

\newcommand{\defn}[1]{\emph{\darkblue #1}}
\newcommand{\defna}[1]{\emph{\darkred #1}}

\newcommand{\defng}[1]{\emph{\darkgreen #1}}

\title[Bounds on Kronecker coefficients]{Durfee squares, symmetric partitions \\ and bounds on Kronecker coefficients}

\author[Igor Pak \, and \, Greta Panova]{ Igor Pak$^\star$ and Greta Panova$^\dagger$}

\thanks{\today}
\thanks{\thinspace ${\hspace{-.45ex}}^\star$Department of Mathematics,
UCLA, Los Angeles, CA~90095.
\hskip.06cm
Email:
\hskip.06cm
\texttt{pak@math.ucla.edu}}

\thanks{\thinspace ${\hspace{-.45ex}}^\dagger$Department of Mathematics,
 USC, Los Angeles, CA~90089.
\hskip.06cm
Email:
\hskip.06cm
\texttt{gpanova@usc.edu}}

\def\nn{\mathbb N}

\def\rr{\mathbb R}

\def\Om{\Omega}

\def\la{\lambda}
\def\ga{\gamma}
\def\si{\sigma}
\def\de{\delta}

\def\al{\alpha}
\def\be{\beta}

\def\ve{\varepsilon}

\def\cP{\mathcal P}

\def\ssu{\subset}

\def\<{\langle}
\def\>{\rangle}

\def\rB{{\text {\rm B} } }

\def\rA{\textup{{\rm \textrm{A}}}}
\def\rAs{\textup{{\rm \textrm{A}{\hskip-0.03cm}$^{\textit{s}}$}}}

\def\rB{\textup{{\rm \textrm{B}}}}
\def\rBfs{\textup{{\rm \textrm{B}{\hskip-0.03cm}$^{\textit{fs}}$}}}

\def\rF{\textup{{\rm \textrm{F}}}}
\def\rK{\textup{{\rm \textrm{K}}}}

\def\rKs{\textup{{\textrm{K}$^{\textit{s}}$}}}
\def\rKfs{\textup{{\textrm{K}$^{\textit{fs}}$}}}

\def\0{{\mathbf 0}}

\def\.{\hskip.06cm}
\def\ts{\hskip.03cm}

\def\nin{\noindent}

\def\LHS{{{left-hand side}}}
\def\RHS{{{right-hand side}{} }}

\begin{document}

\begin{abstract}
We resolve two open problems on Kronecker coefficients
\ts $g(\la,\mu,\nu)$ \ts of the symmetric group.
First, we prove that for partitions \ts $\la,\mu,\nu$ \ts
with fixed Durfee square size, the Kronecker coefficients
grow at most polynomially.  Second, we show that the maximal
Kronecker coefficients \ts $g(\la,\la,\la)$ \ts for
self-conjugate partitions~$\la$ grow superexponentially.
We also give applications to explicit special cases.
\end{abstract}

\maketitle
%\tableofcontents

\section{Introduction}

%\medskip
\subsection{Foreword}\label{ss:intro-Fore}
How do you approach a massive open problem with countless cases to consider?
%where state of the art is meager and the tools are lacking?
You start from the beginning, of course, trying to resolve either
the most natural, the most interesting or the simplest yet  out of reach
special cases.  For example, when looking at the billions and billions of
stars
contemplating the immense challenge
of celestial cartography, you start
with the \emph{closest} (Alpha Centauri and Barnard's Star), the
\emph{brightest} (Sirius and Canopus), or the \emph{most useful}
(Polaris aka North Star),
but not with the
galaxy far, far away.

The same principle applies to the \defna{Kronecker coefficients}
\. $g(\la,\mu,\nu)$.
Introduced by Murnaghan in 1938, they remain among the great mysteries
of Algebraic Combinatorics.  In part due to the fact that they lack a
combinatorial interpretation, even the most basic questions present seemingly
insurmountable challenges, while even the simplest examples are already
hard to compute.  Yet, this should not prevent us from pursuing both.

\smallskip

In our previous paper~\cite{PP_ct}, we briefly surveyed the dispiriting state of art
on Kronecker bounds, and identified two promising problems which are
both interesting, simple looking, yet not immediately approachable with the tools previously used:
\smallskip

$(1)$ \ give upper bounds for \. $g(\la,\mu,\nu)$, where
 $\ts\la,\mu,\nu$ \ts have a \emph{small Durfee square},

\smallskip

$(2)$ \ give lower bounds for the maximal \. $g(\la,\la,\la)$, where \ts $\la$ \ts
is \emph{symmetric}\ts : \. $\la=\la'$.

\smallskip

\nin
We largely resolve both problems, getting estimates
up to a constant in the leading terms of the asymptotics.
For the \defna{small Durfee square problem}~$(1)$, we
employ  symmetric functions technology and obtain new estimates
on the \emph{Littlewood--Richardson coefficients} which
are of independent interest.  For the \defna{fully
symmetric Kronecker problem}~$(2)$,
we use a combinatorial argument based on the \emph{monotonicity property}.
We then use this argument to derive the first nontrivial lower bounds
in several explicit examples.

\medskip

\subsection{Small Durfee square problem}  \label{ss:intro-Durfee}
For a partition \ts $\la \vdash n$, denote by \ts $\ell(\la)$ \ts
the \defn{length} of $\la$, i.e.\ the number of rows in the Young diagram~$\la$.
Denote by \ts $d(\la)$ \ts the \defn{Durfee square size}, i.e.\
the size of the largest square which fits~$\la$. Clearly, \ts $d(\la) \le \ell(\la)$.

The \defn{Kronecker coefficients} \ts $g(\la,\mu,\nu) \in \nn$ \ts
are defined as the structure constants in the ring of characters of \ts $S_n\ts$:
$${\chi^{\mu}} \. \cdot \. \chi^\nu \. = \. \sum_{\la \vdash n} \. g(\la,\mu,\nu) \. \chi^\la\,, \quad
\text{where \ \  $\mu, \. \nu\ts \vdash \ts n$}
$$
and $\chi^\al$ denotes the character of the irreducible representation (Specht module) indexed by the partition $\al$.
Note that \. $g(\la,\mu,\nu)$ \. are symmetric with respect to permutations of three partitions.

It is known that \. $g(\la,\mu,\nu) \le \min\bigl\{f^\la,f^\mu,f^\nu\bigr\}$,
where \ts $f^\al:=\chi^\al(1)$ is the dimension of the Specht module, see \cite{PPY}, but there are no other general
bounds.
On the other hand, for partitions with few rows, we have the following general upper bound:

\smallskip

\begin{thm}[\cite{PP_ct}] \label{t:PP_ct_main}
Let \ts $\la,\mu,\nu\vdash n$, such that \ts $\ell(\la)=\ell$, \ts $\ell(\mu) = m$,
and \ts $\ell(\nu)=r$.  Then:
$$g(\la,\mu,\nu) \ \le \ \Bigg(1\. + \. \frac{\ell m r}{n}\Bigg)^n \.
\Bigg(1 \. + \, \frac{n}{\ell m r}\Bigg)^{\ell m r} .
$$
\end{thm}

In particular, we have:

\smallskip

\begin{cor}[{\rm see~$\S$\ref{ss:Basic-dim}}]\label{c:rows}
Let \ts $\la,\mu,\nu\vdash n$, such that \. $\ell(\la), \ell(\mu), \ell(\nu) \le k$.  Then:
\begin{equation}\label{eq:cor-rows}
g(\la,\mu,\nu) \, \leq \, n^{k^3}\ts.
\end{equation}
\end{cor}

\smallskip

In other words, for partitions with fixed number of rows, the Kronecker coefficients
are bounded polynomially.\footnote{Let us mention \cite{MT} which appeared after the 
first version of this paper, and improves some of our constants from~\cite{PP2,PP_ct} in some special cases 
of partitions with bounded number of rows.}  Recently, we conjectured that the same holds
for partitions with fixed Durfee square size.

\smallskip

\begin{conj}[{\rm \cite[Rem.~5.10]{PP_ct}}]\label{conj:Durfee}
Fix \ts $k\ge 1$ \. and let \. $\la,\mu,\nu\vdash n$, such that \. $d(\la), \ts d(\mu), \ts d(\nu) \. \le \ts k$.
Then \.
$g(\la,\mu,\nu) \ts \leq \ts n^{c}$ \. for some constant \. $c=c(k)>0$.
\end{conj}

\smallskip

The contingency arrays estimates we used in the proof of
Theorem~\ref{t:PP_ct_main} are inapplicable in this case.  Using symmetric functions techniques, here we prove the
conjecture with an explicit constant \ts $c(k)$.

\smallskip

\begin{thm}\label{t:main}
Let \ts $n, k\ge 1$, and let \.  $\la,\mu,\nu\vdash n$, such that \. $d(\la), \ts d(\mu), \ts d(\nu) \. \le \ts k$.  Then:
\begin{equation}\label{eq:main-upper}
g(\la,\mu,\nu) \, \leq \, \frac{1}{k^{8k^2}\ts 2^{8k^3}} \,\, n^{4k^3\ts + \ts 13k^2 \ts + \ts 31k}.
\end{equation}
\end{thm}

\smallskip

Note that the upper bound~\eqref{eq:main-upper}
is slightly weaker than the upper bound in \eqref{eq:cor-rows}, but only by
constant~$4$ is the leading term.  In fact, Corollary~\ref{c:rows}
is crucial for the proof of Theorem~\ref{t:main}.

\smallskip

To appreciate the power of the theorem, compare it with the previous bounds.
For example, when \ts $\la=(m+1, 1^m)$ \ts  is a hook of size \ts $n=2m+1$, so \ts $d(\la)=1$,
the dimension bound is exponential:
$$g(\la,\la,\la) \. \le \. f^\la \. = \. \tbinom{2m}{m} \. = \. \Theta\big(2^{n}/\sqrt{n}\big).
$$
By contrast, Theorem~\ref{t:main}  gives a polynomial upper bound: \. $g(\la,\la,\la)\le n^{40}$.
In fact, it is known that \ts  $g(\la,\la,\la)=1$ \ts  in this special case,
see e.g.~\cite{Rem,Ros}.
% when \ts $\ell m r \le n$, we have $g(\la,\mu,\nu) \le 4^n$ and

Similarly, in \cite[Prop.~5.9]{PP_ct}, we used contingency tables and an ad hoc orbit counting argument
to give a weakly exponential upper bound for the case when \ts $\nu$ \ts is a hook and \ts $\la,\mu$ \ts
are \defng{double hooks}, i.e.\ $d(\la), d(\mu) \le 2$:
$$g(\la,\mu,\nu) \, \le \, n^{450} \ts p(n)^{400} \, = \, e^{\Theta(\sqrt{n})}\ts,
$$
where \. $\la,\mu, \nu \vdash n$ \. and \ts $p(n)$ \ts is the number
of partitions of~$n$. By contrast, Theorem~\ref{t:main} gives a polynomial upper bound:
\. $g(\la,\mu,\nu)\le n^{146}/2^{96}$.\footnote{The constant 146 is likely very far from optimal,
and we make no effort to improve it as we are mostly interested in the asymptotic estimates.}

\smallskip

In the opposite direction, let us show that the bound in \eqref{eq:cor-rows} is tight up to the
lower order terms in the following strong sense.  Let
$$
\rA(n,k) \, := \, \max \. \big\{ \ts g(\la,\mu,\nu) \ : \ \la,\mu,\nu\vdash n \ \
\text{and} \ \ \ell(\la),\ell(\mu),\ell(\nu)\le k \ts \big\}.
$$

\smallskip

\begin{thm}%[{\rm cf.\ Theorem~\ref{t:lower_bound}}]
\label{t:intro-rows-lower}
For all \ts $k\ge 1$, there is a constant \ts $C_k>0$, such that
\begin{equation}\label{eq:main-lower-A}
\rA(n,k) \, \geq \, C_k \. n^{k^3\ts - \ts 3k^2 \ts - \ts 3k \ts + \ts 3}  \qquad
\text{for all \ \ $n\ge 1$\ts.}
\end{equation}
\end{thm}

\smallskip

In other words, the theorem says that for all $n$,
there exist partitions \. $\la,\mu,\nu \vdash n$, such that
\. $\ell(\la), \ts \ell(\mu), \ts \ell(\nu) \le k$ \ts and we have \ts $g(\la,\mu,\nu) \ge$ \ts \RHS of~\eqref{eq:main-lower-A}.
In particular, the theorem implies that the upper bound~\eqref{eq:main-upper} is tight
up to a constant~$4$ in the leading term, cf.~$\S$\ref{ss:finrem-const}.

\medskip

\subsection{Symmetric Kronecker problem}\footnote{Here
``symmetric'' refers to the problem, not the coefficient. The \defng{symmetric
Kronecker coefficients} are defined in \cite{BCI}, and not studied in this paper.}
\label{ss:intro-sym}
Let
$$
\rK(n) \, := \, \max \. \big\{ \ts g(\la,\mu,\nu) \ : \ \la,\mu,\nu\vdash n \ts \big\}
$$
denote the \defn{maximal Kronecker coefficient}.  It was shown by Stanley \cite[slide~44]{Stanley-kron},
that\footnote{Here and throughout the paper we use \ts $f=g\ts e^{-O(n^\al)}$ \ts to mean
that there is a universal constant \ts $c>0$ \ts such that \ts $f \ge g \ts e^{-cn^\al}$ \ts for all \ts $n\ge 1$.
The inequality \ts $f=g-O(n^\al)$ \ts is defined analogously. }
\begin{equation}\label{eq:Kron-max}
\rK(n) \,  = \, \sqrt{n!} \, e^{-O(\sqrt{n})}.
\end{equation}
Later, it was shown in~\cite{PPY} that the maximum can only occur
when all three partitions have \defng{Vershik--Kerov--Logan--Shepp shape},
and their limit curve is of course self-conjugate, see e.g.~\cite{Rom}.  It is thus natural to ask whether
the following  maximal
Kronecker coefficients for the \defn{symmetric} and \defn{fully symmetric} problem have the same asymptotics as~$\rK(n)$:
$$
\aligned
\rKs(n) \, &:= \, \max \. \big\{ \ts g(\la,\la,\la) \ : \ \la\vdash n \ts \big\}, \quad \text{and} \\
\rKfs(n) \, & := \, \max \. \big\{ \ts g(\la,\la,\la) \ : \ \la\vdash n, \ \la=\la' \ts \big\}.
\endaligned
$$
Clearly, \. $\rKfs(n) \le \rKs(n) \le \rK(n)$.

We showed in \cite[$\S$6.3]{PP_ct}, that \.
$\rKs(n) =  e^{\Omega (n^{2/3})}$ \. using an explicit construction
and the asymptotics of plane partitions.\footnote{The statement of Theorem~1.3
in \cite{PP_ct} has an error: we prove the lower bound for \ts $\rKs(n)$,
not \ts $\rKfs(n)$ \ts as claimed in the theorem.}
Although no nontrivial lower bound was known for \ts $\rKfs(n)$, we
(somewhat audaciously) stated:

\smallskip

\begin{conj}[{\rm \cite[Conj.~6.7]{PP_ct}}]\label{conj:rKfs}
$$\log \ts \rKfs(n) \, = \, \frac12 \. n \log n \. - \. O(n).
$$
\end{conj}

\smallskip

Here we give a surprisingly simple proof of a superexponential
lower bound, resolving the problem up to a constant factor
in the logarithm.

\smallskip

\begin{thm}\label{t:Kron-fs-asy}
For all \ts $\ve >0$, we have:
$$
\log \ts \rKfs(n) \, \ge  \, \frac{1}{(16 \. + \. \ve)} \, n \log n \. - \. O(n).
$$
\end{thm}

\medskip

\subsection{Explicit constructions}  \label{ss:intro-explicit}
A key problem we identified in~\cite{PP_ct} is an explicit construction
of partitions \ts $\la,\mu,\nu \vdash n$ \ts with \emph{large} \ts $g(\la,\mu,\nu)$.
Here by \defn{explicit construction} we mean an algorithm which outputs the triple
\ts $(\la,\mu,\nu)$ \ts in \ts poly$(n)$ time. As we mentioned above, in
\cite[$\S$6.2]{PP_ct} we gave an explicit construction of \ts $\la \vdash n$,
such that \ts $\ell(\la) = \Theta(n^{1/3})$ \ts and \ts $g(\la,\la,\la) = e^{\Omega(n^{2/3})}$.

In the \emph{symmetric case} of Kronecker coefficients \ts $g(\la,\la,\la)$ \ts
with \ts $\la=\la'$, until now very little was known.  It was shown in~\cite{BB},
that \ts $g(\la,\la,\la)\ge 1$ \ts for all \ts $\la=\la'$.  In connection with
the \defng{Saxl conjecture}, the
\defn{staircase shape} \ts $\rho_k = (k-1,\ldots,2,1)\vdash \tbinom{k}{2}$ \ts
is especially important, see \cite{Ike,LuS,PPV}.  Unfortunately, the best
lower and upper bounds in this case remain
\begin{equation}\label{eq:rho-bounds}
1\. \le \. g\big(\rho_k\ts, \ts \rho_k\ts, \ts \rho_k\big) \. \le \. f^{\rho_k} \. = \.
\sqrt{n!} \, e^{-O(n)}.
\end{equation}

While we conjecture the asymptotics on the right is the correct estimate, we are
nowhere close to proving this claim (cf.~$\S$\ref{ss:finrem-BBS}).
However, we are able to obtain lower bounds for the other two shapes
considered in~\cite{PPV} in connection with the Saxl conjecture.

\smallskip

\begin{thm}\label{t:explicit}
Let \ts $r\ge 1$, \ts $k=2^{2r+1}$, \ts $n=k^2$ \ts and let \ts
$\de_k:=(k^k)\vdash n$ \ts be the \defn{square shape}.
Then:
$$
g\big(\delta_k\ts, \ts \delta_k\ts, \ts \delta_k\big) \. \ge \. e^{\Omega(n^{1/4})}\ts.
$$
Similarly, let  \ts $r\ge 1$, \ts $k=2^{2r}-1$, \ts $n=3k^2+1$, and let
$$
\tau_k:=\big(3k-1,3k-3,\ldots,k+3,(k+1)^2,(k-1)^2,\ldots,2^2,1^2\big)\. \vdash \. n
$$
be the \defn{caret shape}.  Then:
$$
g\big(\tau_k\ts, \ts \tau_k\ts, \ts \tau_k\big) \. \ge \. e^{\Omega(n^{1/4})}\ts.
$$
\end{thm}

\smallskip

Although both lower bounds are rather weak compared to what we believe to be
the correct asymptotics (see~$\S$\ref{ss:finrem-rho}), these are the first nontrivial
bounds we obtain in this case.  Note that they are weaker than the
bound \. $g(\la,\la,\la) \ge e^{\Omega(n^{2/3})}$ \. from our earlier
explicit construction.

\smallskip

\subsection{Structure of the paper}  We start with a brief
Section~\ref{s:Basic} with definitions and some background.
In the next Section~\ref{s:LR} we give estimates for the
Kostka and Littlewood--Richardson coefficients.  We then proceed to obtain bounds
on the Kronecker coefficients and prove Theorems~\ref{t:main}
and~\ref{t:intro-rows-lower} in Section~\ref{s:Kron}.  We then
prove Theorem~\ref{t:Kron-fs-asy} and give explicit constructions
(Theorem~\ref{t:explicit}) in Section~\ref{s:mono}.  In Section~\ref{s:finrem},
we conclude with final remarks, conjectures and open problems.

\bigskip

\section{Definitions and basic results}\label{s:Basic}

\smallskip

We assume the reader is familiar with the notation and
standard results in the literature, see e.g. \cite{Mac}
and~\cite[$\S$7]{EC2}.  In this section we recall several
useful basic results to help the reader navigate through
the paper.

\smallskip

\subsection{Partitions} \label{ss:Basic-part}
Let \. $\la=(\la_1,\la_2,\ldots,\la_\ell)$ \. be a \defn{partition}
of \defn{size} \. $n:=|\la|=\la_1+\la_2+\ldots+\la_\ell$, where \.
$\la_1\ge \la_2 \ge \ldots \ge\la_\ell\ge 1$.  We write
\. $\la\vdash n$ \. in this case.  As in the introduction,
let \ts $\ell(\la):=\ell$ \ts be the \defn{length} of~$\la$, and let \.
$d(\la)= \max\{k~:~\la_k \ge k\}$ \. denote the \defn{Durfee square size}.

Denote by \ts $p(n)$ \ts the number of partitions \ts $\la\vdash n$.
Let \ts $\la'$ \ts denote the \defn{conjugate partition} of $\la$.
Let \. $\la+\mu$ \. denote the \defn{sum of partitions}: \.
$(\la_1+\mu_1,\la_2+\mu_2,\ldots)$.  Similarly, let \. $\la\cup \mu$ \.
denotes the \defn{union of partitions} defined as \. $\la\cup \mu := (\la'+\mu')'$.
We write \. $\mu \subseteq \la$ \. if \. $\mu_i \le \la_i$ \. for all~$i\ge 1$.
The \defn{skew shape} \ts $\la/\mu$ \ts is the difference of two straight shapes
(Young diagrams).

We use \. $(a^b) = (a,\ldots,a)$, $b$ times, to denote the
\defn{rectangular shape}, and \ts $\rho_\ell = (\ell-1,\ldots,2,1)$ \ts
denotes the \defn{staircase shape}.  Other special partitions include the
\defn{hooks shape} $(k,1^{n-k})$ and the \defn{two-row shape} $(n-k,k)$.

\ytableausetup{boxsize=2ex}

\subsection{Basic inequalities}

Throughout the paper we will use several basic inequalities allowing us to reach the polynomial bounds. By \defn{{\em AM--GM}} we will refer to the Arithmetic vs Geometric mean inequality
$$
\frac{\sum_{i=1}^m x_i}{m} \ \geq \ \left( \prod_{i=1}^m x_i \right)^{1/m}
$$
for nonnegative real numbers $x_i$. In particular, applying this with \/ $x_i =N+1-i$ gives
\begin{equation}\label{eq:binom_exp}
k! \. \binom{N}{k} \ = \ \prod_{i=1}^k (N+1-i) \ \leq \ \left( \frac{\sum_{I=1}^k N+1-i}{k}\right)^k \ = \ \left(N -\frac{k-1}{2} \right)^k
\end{equation}
We will also use the fact that $\left(1 +\frac{x}{n}\right)^n$ is an increasing sequence in $n$ with limit $e^x$.

We will also use the log-concavity of binomial coefficients, namely
\begin{equation}\label{eq:binom_concave}
\binom{x}{k}^2 \ \geq \ \binom{x-1}{k} \binom{x+1}{k} \ \geq \ \cdots \ \ge \ \binom{x-r}{k} \binom{x+r}{k}.
\end{equation}

\smallskip

\subsection{Partition inequalities} We draw partitions as Young diagrams in the English notation.  For example, we use
 $$\ydiagram{4,3,1}$$
for \ts $\la=(4,3,1)$.
 We will use several bounds on the number of partitions which can be easily seen from this graphical representation.

The Young diagram of a partition can be determined by its boundary, which is a North-East monotone lattice path. If the partition has length $k$ and at most $n$ boxes, then, after removing its first column, it fits in an $k\times (n-k)$ rectangle  and we can bound it by the number of lattice paths as
 \begin{equation}\label{eq:partitions_length_bound}
 \# \{ \la \vdash n \, : \, \ell(\la) = k\} \ \leq \ \binom{n}{k}.
 \end{equation}
If $\la\vdash n$ and $\ell(\la) \leq k$, then the number of partitions can be bounded by the (unsorted) number of weak compositions of $n$ into $k$ parts, given by $\binom{n+k-1}{k-1}$ and so
\begin{equation}\label{eq:partitions_general_bound}
%P_k(n) \  = \
\# \{ \la \vdash n \, : \, \ell(\la) \leq k\} \ \leq \ \binom{n+k-1}{k-1} \ = \ O(n^{k-1}),
\end{equation}
 where the last equality holds when \ts $k$ \ts is fixed and \ts $n\to\infty$.

For a partition \ts $\gamma \subset \la$ with $\ell(\la) =k$ and $\la \vdash n$, then $\gamma$ can be viewed as a lattice path inside the \.
 $k \times (n-k+1)$ \. rectangle bounding $\la$.  This gives:
 $$\# \{\ga \, : \, \gamma \subseteq \la \} \ \leq \ \binom{ n-k+1}{k}.$$
For \. $\la \vdash n$ \. and \. $d(\la) \leq k$, we have then:
 \begin{equation}\label{eq:partitions_subset_bound}
 \#\{ \ga \, : \, \gamma \subseteq \la \} \ \leq \ (n/2)^{2k}.
 \end{equation}
 This follows by considering $\gamma$ as a NE lattice path from the lower left corner of $\la$ at $(0,-\ell(\la))$, through a point on the diagonal \. $(i,-i)$ \. with \. $i \in[1,k]$, to $(\la_1,0)$. For \. $\ell(\la) = \ell$, then \. $\la_1 \leq n+1- \ell$, and the total lattice path count gives
 $$
 \sum_{i=1}^k \binom{ \ell }{i} \binom{n+1-\ell}{i} \ \leq \
 \sum_{i=1}^k \binom{ \lfloor n/2 \rfloor }{i} \binom{\lceil n/2 \rceil}{i}
 \ \leq \ \sum_{i=1}^k \frac{ (n/2)^{2i} }{(i!)^2} \ \leq \ (n/2)^{2k}.
 $$
Here the first inequality follows from~\eqref{eq:binom_concave}, and the last inequality is easily seen by induction on~$k$.

 \medskip

\subsection{Kronecker coefficients} \label{ss:Basic-Kron}
Recall an equivalent definition of \defn{Kronecker coefficients}:
$$
g(\la,\mu,\nu) \. = \, \frac{1}{n!} \. \sum_{\si \in S_n} \.  \chi^\la(\si) \ts
\chi^\mu(\si) \ts \chi^\nu(\si) \ts.
$$
From here it is easy to see both the \defn{symmetry} and the \defn{conjugation} properties:
\begin{equation}\label{eq:sym-mono}
g(\la,\mu,\nu) \. = \. g(\mu,\la,\nu) \. = \. g(\la,\nu,\mu) \. = \. \ldots \qquad \text{and} \qquad
g(\la,\mu,\nu) \. = \. g(\la',\mu',\nu).
\end{equation}
We will use the following lesser known \defn{monotonicity property}, which
is an extension of the \defn{semigroup property}, see~\cite{CHM}.

\smallskip

\begin{thm}[\cite{Man}]  \label{t:manivel}
Suppose \. $\al,\be,\ga\vdash m$, such that \. $g(\al,\be,\ga)> 0$.
Then for all \. $\la,\mu,\nu\vdash m$, we have \,
$g(\la+\al,\mu+\be,\nu+\ga)\ts \geq \ts g(\la,\mu,\nu)$.
\end{thm}

\smallskip

\subsection{Dimension bound}\label{ss:Basic-dim}
As in the introduction, denote \. $f^\la=\chi^\la(1)$, which is also the number of Standard Young Tableaux of shape $\la$.
We make frequent use of the \defn{dimension bound} \ts
$g(\la,\mu,\nu)\le f^\la \le \sqrt{n!}$, see e.g.~\cite{PPY}.
For example, we have:

\begin{proof}[Proof of Corollary~\ref{c:rows}]
If \ts $k^3\ge n$, then the result follows from the dimension bound: \ts
$g(\la,\mu,\nu)\le n! \le n^{k^3}$.
If \ts $k^3\le n$ \ts and \ts $k \geq 2$, it follows from~\eqref{t:PP_ct_main}, that
$$g(\la,\mu,\nu) \, \leq \, \left(1  + \tfrac{k^3}{n}\right)^n  \Big( 1 + \tfrac{n}{k^3}\Big)^{k^3}
\, \le \, n^{k^3}.
$$
The last inequality is a consequence of the fact that if $2e\leq a \leq b$ then
$$\left(1 + \frac{a}{b}\right)^b \left( 1+\frac{b}{a}\right)^a \leq e^a \left(1+\frac{b}{a}\right)^a  \leq b^a.$$
Finally, if \ts $k=1$, then all three partitions are hooks, and the result follows from~\cite{Ros}.
\end{proof}

\smallskip

\subsection{Symmetric functions}\label{ss:Basic-sym}
Recall the \defng{homogenous symmetric functions} \. $h_\la$\.,
\defng{elementary symmetric functions} \. $e_\la$,
\defng{monomial symmetric functions} \. $m_\la$ \.   and
\defng{Schur functions} \. $s_\la$. The \defn{Kostka numbers} can be defined as
\begin{equation}\label{eq:h-Kostka}
h_\al \, = \, \sum_{\la\vdash n} \. K_{\la,\al} \. s_\la \quad\text{for all} \quad \al\vdash n.
\end{equation}
More generally, for skew shapes \. $\la/\mu$ \. we have \defn{skew Kostka numbers}
$$
s_\mu \cdot h_\al \, = \, \sum_{\la\vdash |\mu|+|\al|} \. K_{\la/\mu,\al} \. s_\la\..
$$

The Kronecker coefficients can be equivalently defined as follows:
\begin{equation}\label{eq:schur_kron}
s_\la[xy] \, :=  \,  \sum_{\mu,\nu\vdash n} \. g(\la,\mu,\nu) \ts s_\mu(x) \ts s_\nu (y) \quad\text{for all} \quad \la\vdash n\ts,
\end{equation}
where \. $[xy] := (x_1y_1,x_1y_2,\ldots,x_iy_j,\ldots)$ \. denote all pairwise products of variables. The last identity can be written as the \defn{tripple Cauchy identity} given by
\begin{equation}\label{eq:cauchy3}
\sum_{\la,\mu,\nu} g(\la,\mu,\nu) s_\la(x) s_\mu(y) s_\nu(z) = \prod_{i,j,k} \frac{1}{1-x_iy_jz_k}.
\end{equation}

We will also need \defn{Littlewood's identity}~\cite{Lit}:
\begin{equation}\label{eq:littlewood}
s_\la * (s_\al s_\be) \, = \, \sum_{\theta \vdash|\al|, \eta\vdash |\be|} \. c^{\la}_{\theta \eta} \ts \big(s_\al*s_\theta\big) \big(s_\be* s_\eta\big)
\end{equation}
where \ts $c^{\la}_{\mu\nu}$ \ts denote the \defn{Littlewood--Richardson coefficients}, and \ts ``$*$'' \ts denotes
the \defn{Kronecker product of symmetric functions}:
$$
s_\mu \cdot s_\nu \, = \, \sum_{\la\vdash |\mu|+|\nu|} \. c^\la_{\mu\nu} \. s_\la \quad
\text{and} \quad s_\mu * s_\nu \, = \, \sum_{\la\vdash n} \. g(\la,\mu,\nu) \. s_\la\,,\quad\text{for all} \quad \mu, \ts \nu \ts \vdash n\ts.
$$

\medskip

\section{Bounds on Kostka numbers and Littlewood--Richardson coefficients}\label{s:LR}

In this section, we obtain bounds on the Littlewood--Richardson coefficients
in terms of Durfee square size of partitions.  We begin with the \defng{skew Kostka numbers}:

\smallskip

\begin{lemma}\label{l:kostka_bound}
Let \ts $\la/\mu$ \ts be a skew shape, $|\la/\mu|=m$, and let \. $\al\vdash m$.
Suppose \ts $d(\la) \le k$ \ts and \ts $\ell(\al)\leq r$.  Then:
$$
K_{\la/\mu,\al} \, \leq \, 2^{2r}\left( \frac{m}{r} \. + \.\frac{k}{2}\right)^{r(k-1)}.
$$
\end{lemma}
\begin{proof}
Recall the \defng{Pieri rule} for \. $s_\mu \cdot h_a$. We have:
$$K_{\la/\mu,\al} \ =  \ \#\big\{\mu \. \subset \. \la^{(1)} \subset\. \ldots \. \subset \. \la^{(r)} \. = \. \la\big\},$$
where the set has sequences of partitions \. $\la^{(1)},\ldots, \la^{(r)}$, such that \.
$\la^{(i)}/\la^{(i-1)}$ \. is a horizontal strip of size \ts $\al_i$, which can be zero if $\ell(\al)<r$.
Such a strip can have at most $k$ rows of total size~$k$, which are below the diagonal in our English notation, and then at most $k$ rows of total size at most~$\al_i$. The first number is bounded by \ts $2^k$ and the second number is bounded by the number of weak compositions of $\al_i$ into $k$ parts, i.e.\  $\binom{\al_i + k-1}{k-1}$.

Overall, we have:
$$K_{\la/\mu,\al} \, \leq \, 2^{rk} \, \prod_{i=1}^r \. \binom{ \al_i + k-1}{k-1} \, \leq \, \frac{ 2^{rk} }{( (k-1)!)^r} \,
 \prod_{i=1}^r  \left[\. (\al_i+1)\cdots (\al_i+k-1) \right].
$$
Applying the AM--GM inequality as~\eqref{eq:binom_exp} to each product term at the end we get the bound $\left( \alpha_i + \frac{k}{2}\right)^r$. Another application of AM--GM gives
$$\prod_{i=1}^r \left( \alpha_i + \frac{k}{2} \right) \leq \left( \frac{ rk/2  + \sum_i \al_i }{r} \right)^r = \left(\frac{m}{r} +\frac{k}{2} \right)^{r(k-1)}$$
for the big product. Since \. $2^{k-2} \leq (k-1)!$, the first factor is bounded by \ts
$2^{2r}$. Putting it all together we obtain the result.
\end{proof}

\smallskip

\begin{lemma}\label{l:lr_bounds_k_rows}
Let \. $\la \vdash n$, \. $\mu \vdash n-m$ \. and \. $\nu \vdash m$, such that \. $\ell(\la)\ts \le \ts k$. Then:
$$c^{\la}_{\mu\nu} \, \leq \, \left(\frac{2\ts m}{k} \. + \. \frac{k+1}{3} \right)^{\binom{k}{2}}$$
\end{lemma}

\begin{proof}
Recall that \. $c^{\la}_{\mu\nu}$ \. is equal to the number of Littlewood--Richardson tableaux of shape
\ts $\la/\mu$ \ts and type \ts $\nu$. These are characterized by having only $1$'s in the first row,
only $1$'s and $2$'s in the second row, etc. Thus, we have:
$$c^{\la}_{\mu\nu} \ \leq \ \tbinom{\la_2 - \mu_2+1}{1} \. \tbinom{\la_3 - \mu_3+2}{2}\. \cdots \.\tbinom{\la_k - \mu_k+k-1}{k-1}\ts.$$
By the  AM--GM applied to each product term in the numerators of the binomial coefficients, the \RHS is bounded by
$$\frac{1}{ \prod_{i=1}^{k-1} i!} \left( \tfrac{ \sum_{i=2}^k (\la_i-\mu_i)(i-1) \. + \. \binom{k+1}{3}}{\binom{k}{2}}\right)^{\binom{k}{2}}
\ \leq  \ \left(\frac{2\ts m}{k} \. + \. \frac{k+1}{3} \right)^{\binom{k}{2}},$$
which completes the proof.
\end{proof}

\smallskip

\begin{rem}\label{r:LR-upper}{\rm
Lemma~\ref{l:lr_bounds_k_rows} is slightly sharper than the upper bound
in Theorem~4.14 of~\cite{PPY}.  Both upper bounds are the same asymptotically
when the length~$k$ is fixed and \. $n\to\infty$, and match
the lower bound in the same theorem.  The proof of the lemma
is concise and very different from that in~\cite{PPY},
so we included it for completeness.}
\end{rem}

\smallskip

\begin{lemma}\label{l:lr_bounds_k_durfee}
Let \ts $\la\vdash n$ \ts such that \ts $d(\la)\leq k$. Then for every \. $\mu, \nu \subseteq \la$ \.
with \. $|\mu|+|\nu|=n$, we have:
$$c^\la_{\mu\nu} \, \leq \, \left( \frac{n}{k} \. + \.k \right)^{2k^2}.$$
\end{lemma}

\begin{proof}

Let \. $\nu = \alpha \cup \beta'$, where \. $\alpha = (\nu_1,\ldots,\nu_k)$ \. and \. $\beta'=(\nu_{k+1},\ldots)$.
Since \ts $d(\nu)\le d(\la) \leq k$, we have \ts $\ell(\be)\le k$.
Then we have:
$$
c^{\la}_{\mu\nu} \, = \, \langle s_\la, s_\mu \ts s_\nu\rangle \, \leq \, \langle s_\la , s_\mu \ts h_\alpha \ts e_\beta \rangle\ts.
$$
By definition, we have
$$s_\mu \. h_\al \, = \, \sum_{\ga} \. K_{\ga/\mu,\al} \, s_\ga
$$
Multiplying by \ts $e_\be$, expanding in the Schur basis and extracting the \ts $s_\la$ \ts term, we get:
$$
\langle s_\la, s_\mu \ts h_\al \ts e_\be \rangle \, = \, \sum_{\ga} \. K_{\ga/\mu,\al} \. K_{\la'/\ga',\be}\..
$$
Applying the bounds from Lemma~\ref{l:kostka_bound} and using the fact that \. $\ell(\al), \ell(\be)\leq k$, we obtain
$$
c^{\la}_{\mu\nu} \, \leq \, \sum_{\ga\vdash m} \. 2^{4k} \. \left( \frac{m}{k} \. + \. \frac{k}{2}\right)^{k(k-1)}\left( \frac{n-m}{k} \. +\. \frac{k}{2}\right)^{k(k-1)}, \quad \text{where \ \ $m \. := \. |\al|$.}
$$

Observe that the number of partitions \ts $\ga \subset \la$ \ts is bounded above by \. $(n/2)^{2k}$ from inequality~\eqref{eq:partitions_subset_bound}.
Note also that, by the AM--GM inequality, we have:
$$
\left( \frac{m}{k} \. + \. \frac{k}{2}\right)\left( \frac{n-m}{k} \.+\.\frac{k}{2}\right) \ \leq \ \frac14 \left(\frac{n}{k} \. + \. k\right)^2.
$$
Putting everything together, after some cancellations we obtain:
$$c^\la_{\mu\nu} \ \leq \ (n/2)^{2k} \, \frac{2^{4k}}{4^k} \. \left(\frac{n}{k}+k\right)^{2k(k-1)} \ \leq \
\left(\frac{n}{k} +k \right)^{2k^2},
$$
as desired.
\end{proof}

\bigskip

\section{Bounds on Kronecker coefficients}\label{s:Kron}

In this section we prove Theorem~\ref{t:main} and Theorem~\ref{t:intro-rows-lower}, main results of the paper.

\smallskip

\subsection{Upper bounds}\label{ss:Kron-upper}
We approach the problem gradually, starting from the case where two partitions
have bounded number of rows (Lemma~\ref{l:kron_one_durfee}),
then just one (Lemma~\ref{l:kron_two_durfee}), and then in full generality.

\smallskip

Before we proceed with those, we need to add a Lemma complementing the results in~\cite{PP_ct}.

\begin{lemma}\label{l:kron_transp}
Let $\la,\mu,\nu \vdash n$ and $\ell(\la)\leq a$, $\ell(\mu)\leq b$, $\ell(\nu) \leq c$. Then
$g(\la,\mu,\nu') \leq 2^{abc}$.
\end{lemma}

\begin{proof}
From the triple Cauchy identity~\eqref{eq:cauchy3}, applying the involution $\omega$ on the symmetric functions in $z$ we get that $\omega(s_\nu(z)) =s_{\nu'}(z)$ on the \LHS, so
$$\sum_{\la,\mu,\nu} g(\la,\mu,\nu') s_\la(x) s_\mu(y) s_\nu(z) = \prod_{i,j,k} (1+x_iy_jz_k).$$
Expanding both sides in monomials in $x,y,z$ and taking the coefficient at $x^\la y^\mu z^\nu$ on both sides we get
$$g(\la,\mu,\nu') \leq [x^\la y^\mu z^\nu] \prod_{i,j,k=1}^{a,b,c} (1+x_iy_jz_k)$$
as only the variables $x_1,\ldots,x_a, y_1,\ldots, y_b,z_1,\ldots,z_c$ can appear. The \RHS is the generating function of $a \times b \times c$ binary arrays $B$, where for each coordinate $(i,j,k)$ we have a multiplicative weight $(x_iy_jz_k)^{B_{ijk}}$.
Thus, the coefficient on the \RHS above is bounded by the total number of such binary arrays, which is $2^{abc}$. \end{proof}

\begin{lemma}\label{l:kron_one_durfee}
Let \. $\la,\mu,\nu\vdash n$, such that \. $\ell(\la), \ell(\mu) \leq k$ \. and \. $d(\nu) \leq k$.
Then
$$g(\la,\mu,\nu) \, \leq \, 2^{k^3}  n^{k^3+k^2+3k}.$$
\end{lemma}

\begin{proof}
Let \. $\nu = \alpha \cup \beta'$, where \. $\alpha = (\nu_1,\ldots,\nu_k)$ \. and \. $\beta'=(\nu_{k+1},\ldots)$.
Again, since \ts $d(\nu)\le d(\la) = k$, we have \ts $\ell(\be)\le k$.
Let \ts $m=|\al|$, so \ts $n-m=|\be|$

Since \. $s_\al s_\be = s_\nu + \cdots$ \, is a Schur positive sum containing \ts $s_\nu$\ts, we have:
\begin{equation}\label{eq:g1}
g(\la,\mu,\nu) \, = \, \langle s_\la * s_\nu, s_\nu \rangle \, \leq \, \big\langle s_\la *(s_\al s_{\be'}), s_\mu \big\rangle.
\end{equation}
Applying Littlewood's identity~\eqref{eq:littlewood}, we get:
\begin{align}\label{eq:g2}
\big\langle s_\la *(s_\al s_{\be'}), s_\mu \big\rangle \ = \ \sum_{\theta, \eta, \xi, \ga} \. c^{\la}_{\theta  \eta} \. c^{\mu}_{\gamma \xi} \. g(\theta, \al, \gamma) \. g(\eta, \be', \xi)\..
\end{align}
Since \. $c^{\mu}_{\ga\xi}>0$ \. only if \. $\ga,\xi \subset \mu$, and similarly \.
$c^{\la}_{\theta,\eta}>0$ \. only if \. $\theta, \eta \subset \la$,
it follows that all partitions in the \RHS above have length at most~$k$.

We now apply previous results to estimate the \RHS of~\eqref{eq:g2}.
By Corollary~\ref{c:rows}, we have:
$$g(\theta,\al,\gamma)\, \leq \, m^{k^3}.
$$
For the term \. $g(\eta,\be',\xi)$ , note that we  have \. $\ell(\eta), \ell(\xi) \leq k$.  Furthermore,  if \. $g(\eta,\be',\xi)\ne 0$, then by~\cite{Reg}, we must have \.  $\ell(\be') \leq k^2$. Since we also have \. $\ell(\be) = \nu_{k+1} \leq k$, we can apply Lemma~\ref{l:kron_transp} and see that $g(\eta,\beta',\xi) \leq 2^{k^3}$.

Applying Lemma~\ref{l:lr_bounds_k_rows}, we also have upper bounds for the Littlewood--Richardson coefficients involved.
Indeed, denote by \. $r:=\min\{n-m,m\}$. Then for the Littlewood--Richardson coefficients in ~\eqref{eq:g2}, we have:
$$c^{\la}_{\theta  \eta}\., \, c^{\mu}_{\gamma \xi} \
\le \ \left(\frac{2r}{k} \. + \. \frac{k+1}{3}\right)^{\binom{k}{2}}.
$$

Therefore, equations~\eqref{eq:g1} and~\eqref{eq:g2} give
\begin{align*}
g(\la,\mu,\nu) \ \leq \ \sum_{\theta, \eta, \xi, \ga} \. \left(\frac{2r}{k} \. + \. \frac{k+1}{3}\right)^{2\binom{k}{2}} \. 2^{k^3} \. m^{k^3}
\end{align*}
The sum above is over \. $\theta, \eta \subset \la$ \. of sizes \ts $m$, \ts $n-m$,
and over \. $\xi,\ga \subset \mu$ \. of sizes $m$, \ts $n-m$, respectively.
The number of such pairs can be bound  following~\eqref{eq:partitions_general_bound} by $\binom{m+k-1}{k-1}\binom{n-m+k-1}{k-1} \leq_{\eqref{eq:binom_concave}} \binom{n/2+k-1}{k-1}^2$ \. will suffice. We conclude:
\begin{align*}
g(\la,\mu,\nu) \ \leq \ \binom{n/2+k-1}{k-1}^4 \left(\frac{2r}{k} \. + \. \frac{k+1}{3}\right)^{2\binom{k}{2}} 2^{k^3} m^{k^3} \
\leq \, 2^{k^3} n^{k^3+k^2+3k} \ \leq \ C_k  \. n^{(k+1)^3},
\end{align*}
as desired.
\end{proof}

\smallskip

\begin{lemma}\label{l:kron_two_durfee}
Let \. $\la,\mu,\nu \vdash n$, such that \. $d(\mu), \ts d(\nu) \leq k$ \. and \. $\ell(\la) \leq k$. Then
$$g(\la,\mu,\nu) \, \leq \, \frac{1}{k^{2k^2}} \,\ts n^{2k^3 \ts +\. \frac{9}{2} k^2 \ts + \. \frac{19}{2} k}\,.
$$
\end{lemma}

\begin{proof}
As before, let \. $\mu = \alpha \cup \beta'$, where \. $\alpha=(\mu_1,\ldots,\mu_k)$ \. and \.
$\beta = (\mu_{k+1},\ldots)'$. Let \. $m=|\al|$, \. $n-m=|\be|$, and let \. $r:=\min\{m,n-m\}$.
Since \. $s_\al s_\be = s_\mu + \cdots$ \. is a Schur positive sum containing \. $s_\mu$\ts,
we have:
\begin{equation}\label{eq:eq-Kron1}
g(\la,\mu,\nu) \, = \, \langle s_\la * s_\mu, s_\nu \rangle \, \leq \, \big\langle s_\la *(s_\al \ts s_{\be'}), s_\nu \big\rangle.
\end{equation}
Applying Littlewood's identity~\eqref{eq:littlewood}, we get
\begin{equation}\label{eq:eq-Kron2}
\big\langle s_\la *(s_\al \ts s_{\be'}), s_\nu \big\rangle \ = \
\sum_{\theta, \eta,\ga,\xi} \. c^{\la}_{\theta \eta} \. c^{\nu}_{\gamma\xi} \. g(\theta, \al, \gamma) \. g(\eta, \be', \xi)\ts.
\end{equation}

We will bound the terms in the \RHS of~\eqref{eq:eq-Kron2}.
For partitions \. $\theta, \eta$ \. such that \. $c^{\la}_{\theta,\eta}>0$ \.
we must have \. $\theta,\eta \subset \la$, and so \. $\ell(\theta),\ell(\eta) \leq k$.
By Lemma~\ref{l:lr_bounds_k_rows}, we thus have:
$$c^{\la}_{\theta\eta} \, \leq \, \left(\frac{2r}k \ts + \ts \frac{k+1}3\right)^{\binom{k}{2}}.
$$
On the other hand, since we only select partitions \. $\gamma, \xi$, for which \. $c^{\nu}_{\ga\xi}>0$,
then we must have \. $\ga, \xi \subset \nu$, and so \. $d(\ga), d(\xi) \le k$.
By Lemma~\ref{l:lr_bounds_k_durfee}, we thus have:
$$
c^{\nu}_{\ga\xi} \, \leq \, \left(\frac{n}k \ts + \ts k\right)^{2k^2}.
$$

For the Kronecker coefficients in the summation, by Lemma~\ref{l:kron_one_durfee}, we have:
$$g(\theta,\al,\ga) \, \leq \, 2^{k^3} m^{k^3+k^2+3k}.
$$
Similarly, we have:
$$g(\eta,\be',\xi) \, = \, g(\eta,\be,\xi') \, \leq  \, 2^{k^3} (n-m)^{k^3+k^2+3k}.
$$

Now, the summation in the \RHS of~\eqref{eq:eq-Kron2}, we bound the number of pairs of partitions \. $\theta, \eta$ following inequalities~\eqref{eq:partitions_length_bound}\. and~\eqref{eq:binom_concave}
by $\binom{n/2+k-1}{k-1}^2$, and the number of partitions \. $\ga,\xi$ \. by \. $(n/2)^{2k}$ from~\eqref{eq:partitions_subset_bound}.

Combining \eqref{eq:eq-Kron1}, \eqref{eq:eq-Kron2} and the upper bounds above, we conclude:
\begin{align*}
g(\la,\mu,\nu) \ & \leq \ \binom{n/2+k-1}{k-1}^2 \left(\frac{n}{2}\right)^{4k} 2^{2k^3} \left(\frac{n}{k} \ts + \ts k\right)^{2k^2} \left(\frac{2r}k \ts + \ts \frac{k+1}3\right)^{\binom{k}{2}}m^{k^3+k^2+3k} (n-m)^{k^3+k^2+3k}\\
& \leq \ \frac{1}{k^{2k^2}} \, n^{2k^3 \ts + \ts 2k^2 \ts + \ts  6k \ts + \ts  6k \ts + \ts  2k^2 \ts + \ts  \binom{k}{2}} \ = \ \frac{1}{k^{2k^2}} \, n^{2k^3\ts + \ts \frac{9}{2}k^2\ts + \ts \frac{23}{2}k},
\end{align*}
where the constant factors involving \ts $k$ \ts are altogether bounded by \. $k^{-2k^2}$.
\end{proof}

\smallskip

\begin{proof}[Proof of Theorem~\ref{t:main}]
We use the same setup as in the proofs of Lemma~\ref{l:kron_two_durfee}, where $\mu = \al \cup \be'$ and $m=|\al|$. We have:
\begin{align*}
g(\la,\mu,\nu) \ \leq \ \big\langle s_\la *(s_\al s_{\be'}), s_\nu \big\rangle \ = \
\sum_{\theta, \eta,\ga,\xi} \, c^{\la}_{\theta \eta} \, c^{\nu}_{\gamma \xi} \,\. g(\theta, \al, \gamma) \, g(\eta, \be', \xi).
\end{align*}
Again, we must have \. $d(\theta)$, \ts $d(\eta)$, \ts $d(\gamma)$, \ts  $d(\xi)\ts \le \ts k$.  Thus,
we can apply the upper bounds on the Kronecker coefficients from Lemma~\ref{l:kron_two_durfee}, and on the Littlewood--Richardson coefficients
from Lemma~\ref{l:lr_bounds_k_durfee}.  Bounding the number of partitions $\theta, \eta, \gamma,\xi$ from~\eqref{eq:partitions_subset_bound} by $(n/2)^{2k}$, we obtain
\begin{align*}
g(\la,\mu,\nu) \ & \leq \ \sum_{\theta, \eta,\ga,\xi} \.\left(\frac{n}{k} \ts + \ts k\right)^{4k^2}\frac{1}{k^{4k^2}}
\, m^{2k^3+\frac92 k^2 \ts + \ts\frac{19}{2}k} \. (n-m)^{2k^3\ts + \ts\frac92 k^2 \ts + \ts \frac{23}{2}k}\\
& \leq \ \left( \frac{n}{2} \right)^{8k} \frac{1}{k^{4k^2}} \left(\frac{n}{k}+k\right)^{4k^2} m^{2k^3+\frac92 k^2 + \frac{23}{2}k} (n-m)^{2k^3 \ts + \ts \frac92 k^2  \ts + \ts  \frac{23}{2}k} \\
& \leq  \ \frac{1}{k^{8k^2}}\left( \frac{n}{2} \right)^{8k} \left( \frac{n}{2} \right)^{4k^3 \ts + \ts 9k^2 \ts + \ts 23k}  n^{4k^2} \ = \ \frac{1}{k^{8k^2}\ts 2^{8k^3}} \, n^{4k^3 \ts + \ts 13k^2 \ts + \ts 31k},
\end{align*}
which completes the proof.
\end{proof}

\smallskip

\subsection{Proof of Theorem~\ref{t:intro-rows-lower}} \label{ss:Kron-lower}
Let $n=ak$. Combining~\eqref{eq:h-Kostka} and \eqref{eq:schur_kron}, we have the following identity:
$$
h_{a^k}[xy] \ = \ \sum_{\la\vdash n, \. \ell(\la)\leq k} \. K_{\la,a^k} \. s_{\la}[xy] \ = \
\sum_{\la,\mu,\nu\vdash n} \, K_{\la,a^k} \. g(\la,\mu,\nu) \. s_\mu(x) \. s_\nu(y)\ts.
$$
Let \. $x=(x_1,\ldots,x_k)$ \. and \. $y=(y_1,\ldots,y_k)$, so all partitions
in the above identity have lengths bounded by~$k$. Compare the coefficients at \.
$m_{a^k}(x)\cdot m_{a^k}(y)$ \. on both sides, where \ts $m_\al$ \ts are
\defng{monomial symmetric functions}. We then have:
\begin{align}\label{eq:h_g_ct}
[x_1^a\ldots x_k^a \. y_1^a \ldots y_k^a] \ts h_{a^k}[xy] \ = \
\sum_{\la,\mu,\nu \vdash n} \, g(\la,\mu,\nu) \. K_{\la,a^k} \. K_{\mu,a^k} \. K_{\nu, a^k}\..
\end{align}
Consider the term on the \LHS. We have that
$$
h_a[xy] \ = \ \sum_{M} \, \prod_{i,j} \. (x_i y_j)^{M_{ij}} \ = \ \sum_M \. x^{row(M)} \. y^{col(M)}
$$
is the generating function for \defng{contingency tables} \ts $M=\big(M_{ij}\big)$ \ts
with respect to their row and column sums.
Since \. $h_{a^k} = (h_a)^k$, we conclude that the coefficients at \.
$x_1^a\cdots x_k^a \. y_1^a\cdots y_k^a$ \. are equal to the number of
\defng{3-dim contingency arrays} \ts $A$ \ts with all \defng{2-dim marginals} equal to~$a$.
We refer to \cite{PP_ct} for precise definitions and further details.

Geometrically, these contingency arrays~$A$ are integer points in a three-way
\defng{transportation polytope} \. $T_k(m)\ssu \rr^{k^3}$ \. such that \.
$\dim T_k(a) \ts = \ts (k^3 -3k)$.  By the Ehrhart theory for rational
polytopes (see e.g.~\cite[$\S$3.7]{BR}), the number of such points is given by
a quasipolynomial in~$a$ of degree \ts $(k^3-3k)$.  Thus there exists a constant \ts $G_k>0$
(see also~$\S$\ref{ss:finrem-barv}),
such that
\begin{align}\label{eq:ct_bound}
\big[x_1^a\ldots x_k^a \ts y_1^a \ldots y_k^a\big] \ts h_{a^k}[xy]  \, \geq \, G_k \ts a^{k^3-3k}.
\end{align}

On the other hand, in~\eqref{eq:h_g_ct} we have \. $K_{\la,a^k} \leq a^{k^2-k}$ \.
by Lemma~\ref{l:kostka_bound}, and a similar bound for the other Kostka numbers.  We conclude:

\begin{equation}\label{eq:sum_g_bound}
\aligned
& \sum_{\la,\mu,\nu \ts \in \ts \cP_k(n)} \. g(\la,\mu,\nu) \. K_{\la,a^k} \. K_{\mu,a^k} \. K_{\nu, a^k} \ \leq  \
a^{3k^2-3k} \. \sum_{\la,\mu,\nu \ts \in \ts  \cP_k(n)} \. g(\la,\mu,\nu) \\
&\hskip1.2cm \leq \ \big|\cP_k(n)\big|^3 \. a^{3k^2-3k} \. \max_{\la,\mu,\nu \ts \in \ts  \cP_k(n)} \ts g(\la,\mu,\nu) \ \leq \ a^{3k^2-3} \. \max_{\la,\mu,\nu \ts \in \ts  \cP_k(n)} \ts g(\la,\mu,\nu),
\endaligned
\end{equation}
where \. $\cP_k(n)=\bigl\{\la\vdash n \. :\. \ell(\la)\le k\bigr\}$,
so that \. $|\cP_k(n)|=O\big(n^{k-1}\big)$.  Comparing the inequalities
from~\eqref{eq:ct_bound} and~\eqref{eq:sum_g_bound}, we obtain
$$\max_{\la,\mu,\nu \ts \in \ts  \cP_k(n)} \. g(\la,\mu,\nu) \,
\geq \, G_k \. a^{k^3 \ts - \ts 3k \ts - \ts 3k^2 \ts + \ts 3},
$$
as desired. \qed

\bigskip

\section{Kronecker bounds via the monotonicity property}\label{s:mono}

\subsection{Bounds for the symmetric Kronecker problem}  \label{ss:mono-sym}
For all \ts $n, k\ge 1$, define
$$
\aligned
\rAs(n,k) \, &:= \, \max \. \big\{ \ts g(\la,\la,\la) \ : \ \la\vdash n, \ \ell(\la) \le k \ts \big\}, \quad \text{and} \\
\rBfs(n,k) \, & := \, \max \. \big\{ \ts g(\la,\la,\la) \ : \ \la\vdash n, \ \la=\la',  \ d(\la) \le k \ts \big\}.
\endaligned
$$
Clearly, \. $\rAs(n,k)\le \rKs(n)$ \. and \. $\rBfs(n,k) \le \rKfs(n)$.

\smallskip

\begin{lemma}\label{l:Kron-sym}
For all $n\ge 1$, we have:
\begin{equation}\label{eq:Kron-sym}
\rKs(3n) \, \ge \, \rK(n) \quad \text{and} \quad \rAs(3n,k)\ge \rA(n,k)\ts.
\end{equation}
\end{lemma}

\begin{proof} Let \. $g(\al,\be,\ga)=\rK(n)$, for some \ts $\al,\be,\ga\vdash n$.
Let \. $\la:=(\al+\be+\ga)\vdash 3n$. By the symmetry property~\eqref{eq:sym-mono} and
monotonicity property (Theorem~\ref{t:manivel}),
we have:
$$
\aligned
\rKs(3n) \, & \ge \,
g(\la,\la,\la) \, = \, g\big(\al+\be+\ga, \. \be+\ga+\al, \. \ga+\al+\be\big)  \\
\, & \ge \, \max \big\{\ts g(\al,\be,\ga), \, g(\be+\ga,\. \ga+\al,\. \al+\be) \ts\big\} \\
\, & \ge \, \max \big\{\ts g(\al,\be,\ga), \, g(\be,\ga,\al), \, g(\ga,\al,\be)\ts\big\} \, = \, \rK(n).
\endaligned
$$
This proves the first inequality in~\eqref{eq:Kron-sym}.  The second inequality follows
verbatim the argument above and the fact that \ts $\ell(\la) = \max\{\ell(\al),\ell(\be),\ell(\ga)\}$.
\end{proof}

\smallskip

\begin{cor}
For all \ts $k\ge 1$, there is a constant \ts $C_k>0$, such that
\begin{equation}\label{eq:main-lower-As}
\rAs(n,k) \, \geq \, C_k \. n^{k^3\ts - \ts 3k^2 \ts - \ts 3k\ts + \ts 3} \qquad
\text{for all \ \ $n\ge 1$\ts.}
\end{equation}
\end{cor}

\begin{proof}
Combining Theorem~\ref{t:intro-rows-lower} and Lemma~\ref{l:Kron-sym}, we obtain the
result for \ts $3\ts | \ts n$.  For general~$n$, note that \ts $g(\la+1,\mu+1,\nu+1)\ge g(\la,\mu,\nu)$,
again by the monotonicity property.  Thus, we have \. $\rAs(n+1,k) \ge \rAs(n,k)$.  This completes the proof.
\end{proof}

\smallskip

\begin{lemma}\label{l:Kron-fsym}
For all \ts $n, k \ge 1$, we have:
\begin{equation}\label{eq:Kron-fsym}
\rBfs(4n+k^2,k) \, \ge \, \rAs(n,k)\ts.
\end{equation}
\end{lemma}

\begin{proof}
Let \. $g(\al,\be,\ga)\ge 1$, for some \ts $\al,\be,\ga\vdash n$ \ts
such that \ts $\ell(\al), \ts \ell(\be), \ts \ell(\ga) \ts\le k$.
Recall from the introduction that \. $g\big(\de_k,\de_k,\de_k\big)\ge 1$,
where \ts $\de_k=(k^k)$ \ts is the square shape.
By the repeated alternating application of the monotonicity property
(Theorem~\ref{t:manivel}) and symmetry/conjugation properties~\eqref{eq:sym-mono},
we have the following \defn{long formula}:
$$\aligned
g\big(\al,\be,\ga\big) \ & \le \ g\big(\de_k+\al,\de_k+\be,\de_k+\ga\big) \ = \ g\big(\de_k\cup\al',\de_k\cup\be',\de_k+\ga\big) \\
& \le \ g\big((\de_k+\be) \cup\al',(\de_k+\ga)\cup\be',\de_k+\ga+\al\big) \\
& \hskip1.3cm \ = \ g\big((\de_k+\be) \cup\al',(\de_k+\be)\cup\ga',\de_k\cup (\ga+\al)'\big)\\
& \le \ g\big((\de_k+\be+\ga) \cup\al',(\de_k+\be+\al)\cup\ga',(\de_k+\be)\cup (\ga+\al)'\big) \\
& \hskip1.3cm \ = \ g\big((\de_k+\al) \cup(\be+\ga)',(\de_k+\ga)\cup(\be+\al)',(\de_k+\be)\cup (\ga+\al)'\big) \\
& \le \ g\big((\de_k+\al+\be) \cup(\be+\ga)',(\de_k+\ga+\al)\cup(\be+\al)',(\de_k+\be+\ga)\cup (\ga+\al)'\big).
\endaligned
$$
See Figure~\ref{f:long-formula} below for an illustration.

Now, let \. $\al=\be=\ga$ \. and suppose \. $g(\al,\al,\al)=\rAs(n,k)$.  From above,
$$
g\big(\al,\al,\al\big) \ \le \ g\big((\de_k+2\al) \cup(2\al)', (\de_k+2\al) \cup(2\al)', (\de_k+2\al) \cup(2\al)'\big) \ = \
g(\mu,\mu,\mu),
$$
where \. $\mu := \ts \big(\de_k + 2\al\big) \cup (2\al)' \ts \vdash \ts 4n+k^2$.
Since \ts $\mu=\mu'$ \ts and \ts $d(\mu)\le k$, we conclude:
$$
\rBfs(4n+k^2,k) \,\ge \, g(\mu,\mu,\mu) \, \ge \, g(\al,\al,\al) \, = \, \rAs(n,k),
$$
as desired.
\end{proof}

\begin{figure}[hbt]
 \begin{center}
  \includegraphics[height=8.cm]{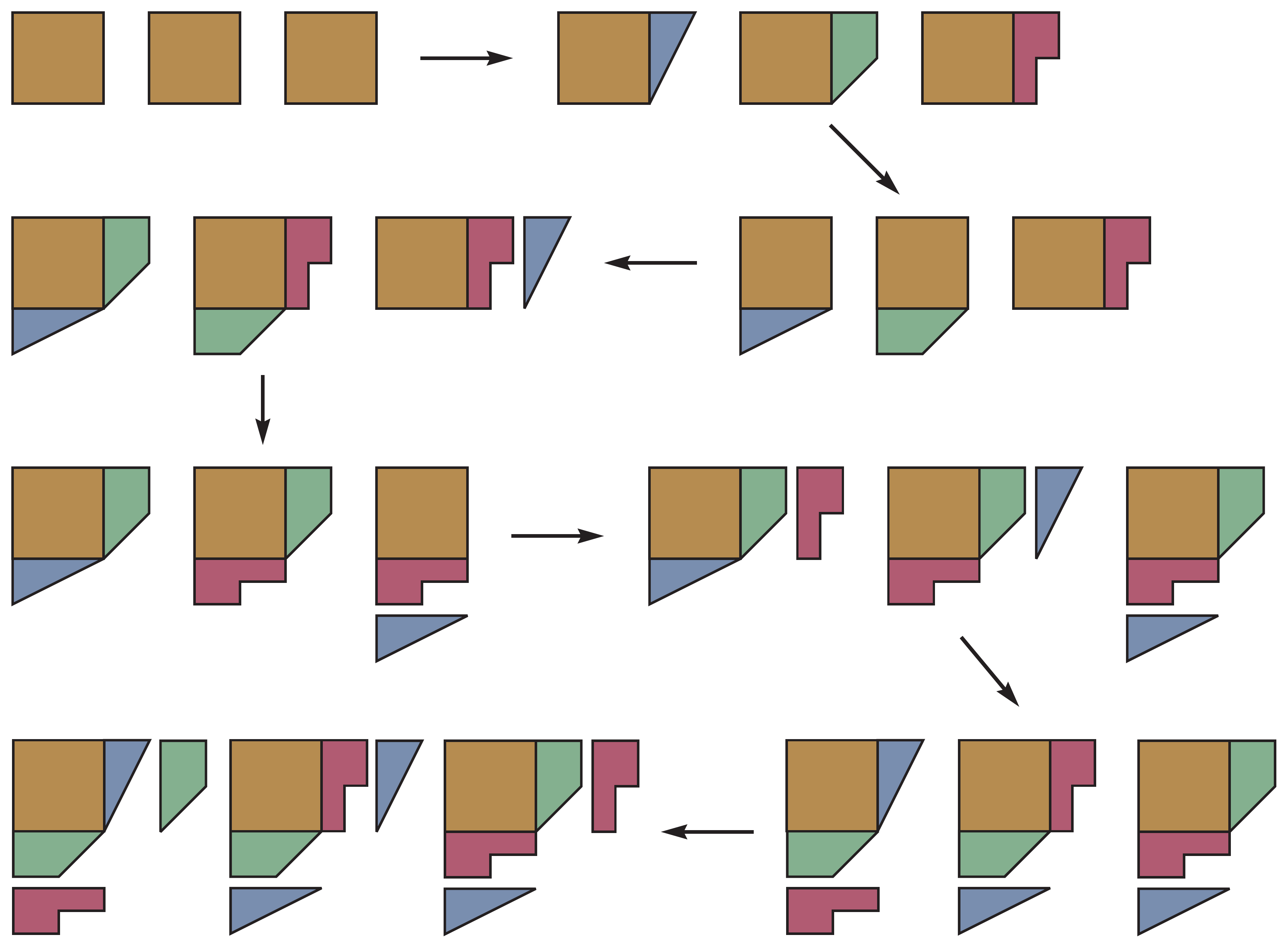}
   \label{f:long-formula}
 \end{center}
 \caption{Graphic rendering of the long formula in the proof of Lemma~\ref{l:Kron-fsym}. }
\end{figure}

\smallskip

\begin{cor}\label{c:Bfs}
For all \ts $k\ge 1$, there is a constant \ts $D_k>0$, such that
\begin{equation}\label{eq:main-lower-Bfs}
\rBfs(n,k) \, \geq \, D_k \. n^{k^3 \ts - \ts 18 k^2 \ts + \ts 102 k \ts - \ts 182} \qquad
\text{for all \ \ $n\ge 1$\ts.}
\end{equation}
\end{cor}

\begin{proof}
Combining Lemma~\ref{l:Kron-sym} and Lemma~\ref{l:Kron-fsym}, we have
\begin{equation}\label{eq:Kron-fsym-cor}
\rBfs(12n+k^2,k) \, \ge \, \rA(n,k)\ts.
\end{equation}
Now Theorem~\ref{t:intro-rows-lower} implies the result for \ts $12\ts | \ts (n-k^2)$.

For general~$n$, note that in the proof of Lemma~\ref{l:Kron-fsym}, we can use \.
$$\mu \. := \. \big(\de_{k+c} + 2\al\big) \cup (2\al)' \ts \vdash \ts 4n+(k+c)^2 \quad
\text{for all \ $c\ge 1$ \ and \ $\ell(\al) \leq k$.}
$$
We can also replace \ts $\delta_{k+c}$ \ts with a chopped symmetric square  by removing a symmetric partition of size \. $t\in \{0,1,3,4,\ldots,10,11,14\}$ from its bottom right corner. Since \ts $g(\la,\la,\la)\geq 1$ \ts for all \ts $\la=\la'$, see~\cite{BB,PPV}, we can then repeat the steps in the proof of Lemma~\ref{l:Kron-fsym} with the chopped square instead of the \ts $\delta_k$. This allows us to construct symmetric partitions $\mu$ of any size modulo~$12$, and note that \ts $c=5$ \ts suffices.

We conclude that \. $g(\mu,\mu,\mu) \geq A(n,k) $ \. for some \.
$\mu \vdash 12n+(k+c)^2-t$ \. with \. $k \leq d(\mu) \leq k+c$.
This implies that \. $\rBfs(n,k) \ts\geq \ts A\big(\lfloor (n-k^2)/12 \rfloor, k-5\big)$, and the bound follows.
\end{proof}

\smallskip

\begin{rem}\label{r:Bfs}{\rm
A curious Conjecture~5.12 in~\cite{BRR}, claims that
$$g\big((\la+1)\cup 1, \ts (\mu+1)\cup 1, \ts (\nu+1)\cup 1\big) \, \ge \, g(\la,\mu,\nu) \quad
\text{for all} \ \. \la,\mu,\nu \vdash n.
$$
This would imply that \ts $\rBfs(n+2,k)\ge \rBfs(n,k)$ \ts and improve
the lower order terms in \eqref{eq:main-lower-Bfs}. }
\end{rem}

\smallskip

\begin{proof}[Proof of Theorem~\ref{t:Kron-fs-asy}]
Fix \ts $\ve>0$ \ts and let \. $k=(2+\ve)\sqrt{n}$.  Let us show that
for sufficiently large \ts $n> N(\ve)$, we have \ts $\rA(n,k) = \rK(n)$.
We follow~\cite{PPY} in our presentation.

Recall that a sequence \ts $\{\la\}$ \ts of partitions is called \defn{Plancherel}
if \. $f^\la = \sqrt{n!} \. e^{-O(\sqrt{n})}$.  Suppose that \ts $g(\la,\mu,\nu)=\rK(n)$.
By Stanley's theorem \eqref{eq:Kron-max} and the dimension bound, we have:
$$\sqrt{n!} \. e^{-O(\sqrt{n})} \. = \. \rK(n) \. = \. g(\la,\mu,\nu) \. \le \. f^\la \. \le \. \sqrt{n!}$$
We conclude that all three sequences
\ts $\{\la\}$, \ts $\{\mu\}$ \ts and \ts $\{\nu\}$ \ts achieving the maximum \. $g(\la,\mu,\nu)$ \.
are
Plancherel.  In fact, we can even fix two of these three sequences
(see \cite[Thm~1.4]{PPY}).

Now, by the VKLS Theorem (see \cite[Thm~1.3]{PPY}), all three sequences must
have \defng{VKLS shape}.  Without stating it explicitly, it follows from the
definition that \.
$$
\ell(\la), \. \ell(\mu), \. \ell(\nu) \, \le \,  2\sqrt{n} \. +  \. O(n^{1/6})
\. \le \. (2+\ve)\sqrt{n} \. = \. k\ts,
$$
for $n$ large enough.  Thus we have \ts $\rA(n,k) = \rK(n)$ \ts in that case.
By~\eqref{eq:Kron-fsym-cor}, we conclude:
%\begin{equation}\label{eq:Kron-fsym-proof}
$$
\rKfs\big((16+3\ve)n\big) \, \ge \, \rKfs(12n+k^2) \, \ge \, \rBfs\big(12n+k^2,k\big) \, \ge \, \rA(n,k) \, = \, \rK(n) \, = \, \sqrt{n!} \. e^{-O(\sqrt{n})}\ts,
$$
%\end{equation}
for \ts $k=(2+\ve)\sqrt{n}$ \ts as above and \ts $n$ \ts large enough.
Taking logs on both sides implies the result.
\end{proof}

\medskip

\subsection{Proof of Theorem~\ref{t:explicit}}  \label{ss:mono-explicit}

We now use the iterated conjugation trick in the proof of Lemma~\ref{l:Kron-fsym}
to give the first nontrivial lower bound for \ts $g\big(\de_k,\de_k,\de_k\big)$.

We start with \cite[Thm~1.2]{PP2}, which gives for $k=2s^2$ that
$$
g\big((2s)^{2s},(2s)^{2s},k^2\big) \, \ge \, C \. 2^{2s} \. (2s)^{-9/2}>0
$$
for some universal constant \. $C > 0.004$.

Observe that by conjugating two partitions we get
$$
g\big((2s)^{2s},(2s)^{2s},2^k\big) \. = \. g\big((2s)^{2s},(2s)^{2s},k^2\big) \. > \. 0.
$$
Let \ts $s:=2^r$. We can repeatedly apply the combination of monotonicity and conjugation,
from \ts $4m=k=2s^2$,  to get
\begin{equation*}%\label{eq:explicit-4square}
\aligned
& g\big((4m)^{4m},\ts (4m)^{4m}, \ts (4m)^{4m}\big) \ \ge \
g\big((2m)^{4m},\ts (2m)^{4m}, \ts (2m)^{4m}\big) \ =
\\
 & \hskip.8cm = \
g\big((4m)^{2m},\ts (4m)^{2m}, \ts (2m)^{4m}\big) \ \ge \
g\big((2m)^{2m},\ts(2m)^{2m}, \ts (m)^{4m}\big) \ \ge \
\\
& \hskip1.6cm \ge \ \ldots \ \geq \ g\big((2s)^{2s},(2s)^{2s},k^2\big) \ > \ C \. 2^{2s} \. (2s)^{-9/2}.
\endaligned
\end{equation*}
Since \. $s= \sqrt{2k} = \sqrt{2} \. n^{1/4}$, we obtain the first part of Theorem~\ref{t:explicit}.

For the caret shape \ts $\tau_k$, note that
$$
\tau_k \, = \, \big(\de_{k+1} + 2\ts \rho_{k}\big) \. \cup \. 2\ts \rho_{k}\..
$$
In notation of the proof of  Lemma~\ref{l:Kron-fsym}, let \. $\al:= \rho_k$ \.
and recall from \eqref{eq:rho-bounds} that \ts $g(\al,\al,\al)>0$.  From the long formula in the proof, we have:
$$\aligned
g\big(\tau_k, \tau_k, \tau_k\big) \ & = \
g\big((\de_{k+1} +2\al) \cup 2 \al,(\de_{k+1} +2\al) \cup 2 \al,(\de_{k+1} +2\al) \cup 2 \al\big) \\
& \ge \ \max \big\{\ts g(\al,\al,\al), \. g(\de_{k+1},\de_{k+1},\de_{k+1}) \ts\big\}.
\endaligned
$$
Now the second part of Theorem~\ref{t:explicit} follows from the first part.  \qed

\bigskip

\section{Final remarks and open problems}\label{s:finrem}

\smallskip
\subsection{} \label{ss:finrem-durfee}
The importance of Durfee square size in connection with the \defng{vanishing
of Kronecker coefficients} (i.e., whether they are nonzero),
has long been understood in the literature.  We refer to
\cite[$\S$3]{BB} and~\cite{Dvir,Reg} for some notable examples.

\subsection{} \label{ss:finrem-special}
There are many special cases
of partitions with small Durfee square size (at most three),
where the Kronecker coefficients are computed exactly, see
e.g.\ \cite{BWZ,BOR,RW,Tew}.  In all these cases the
Kronecker coefficients are bounded by a constant.  This is in
sharp contrast with examples in~\cite{BV,MRS} and
our lower bound in Theorem~\ref{t:intro-rows-lower},
suggesting that being bounded is a small numbers phenomenon.

\subsection{} \label{ss:finrem-stab}
Recall Murnaghan's \defng{stability property}: \ts
the sequence \. $(a_0,a_1,a_2,\ldots)$ \. defined as
$$
a_d \. = \. a_d(\la,\mu,\nu) \, := \, g\big(\la+(d),\ts \mu+(d),\ts \nu+(d)\big)
$$
is increasing and bounded.  This phenomenon was recently
generalized by Stembridge
$$
a_d \. = \. a_d(\la,\mu,\nu; \. \al,\be,\ga) \, := \, g\big(\la\ts+ \ts d\ts\al, \ts \mu\ts+ \ts d\ts \be,\nu\ts+ \ts d \ts\ga\big)
\quad \text{for \ \, $g(\al,\be,\ga)\ts = \ts 1$. }
$$
The nondecreasing of \ts $\{a_d\}$ \ts
follows from the monotonicity property (Theorem~\ref{t:manivel}),
while boundedness was proved by Sam and Snowden~\cite{SS}.

It is known that \. $a_d$ \. are a quasi-polynomial in~$d$ \cite{Man},
see also \cite{BV,Man-2}.
In view of Corollary~\ref{c:rows} and Theorem~\ref{t:intro-rows-lower},
it would be interesting to give a
combinatorial description of the degree of these quasi-polynomials.
Let us mention that \ts
$a_d\ge d+1$ \ts for all \ts $g(\al,\be,\ga)>1$ \ts \cite[Prop.~3.2]{Ste}.

Similarly, one can consider more general families of Kronecker coefficients
$$
b_d \. = \. b_d(\la,\mu,\nu; \. \al,\be,\ga; \. \zeta,\xi,\eta) \, := \, g\big((\la\ts+ \ts d\ts\al) \cup (d\ts \zeta)',
\ts (\mu\ts+ \ts d\ts \be) \cup (d \ts \xi)', \ts(\nu \ts+ \ts d \ts\ga) \cup (d \ts \eta)'\big),
$$
where \. $g(\al,\be,\ga), \ts g(\zeta,\xi,\eta)\ge 1$.
In view of Theorem~\ref{t:main}, we conjecture that \. $b_d$ \. are again
quasi-polynomial in~$d$, for $d$ large enough.  In a special case when
\. $\al=\be =\ga=\zeta=\xi=\eta=1$, this is the \defng{hook stability} \ts
introduced in~\cite{BRR}.
Note that even characterizing
the cases when these quasi-polynomials are nonzero is quite challenging
%  e.g.\ $g\big(\la\cup 1^d, \mu \cup 1^d, \nu \cup 1^d\big)= 0$ \. for \ts $d$ \ts large enough,
% for all \. $\la,\mu,\nu$
(cf.\ \cite[$\S$6.2]{BRR}).

A more general approach to stability for Lie groups can be found in~\cite{Par}.

\smallskip
\subsection{} \label{ss:finrem-barv}
One can give explicit lower bounds on the constants \ts $C_k$ \ts
in~\eqref{eq:main-lower-A}. For that, in notation of the proof
of Theorem~\ref{t:intro-rows-lower}, we need to use the
integral volume \ts $G_k$ \ts of the
\defng{three-way transportation polytopes} \ts $T_k(1)$.
These polytopes are highly symmetric, so the lower bounds
are especially simple and can be found in~\cite{Ben},
see also~\cite{Bar} for a survey.  See also a recent explicit
lower bound in~\cite[Ex.~2.6]{BR}, on the (usual) volume of \ts $T_k(1)$.
To put these bounds into context, recall the natural upper bound
used in~\cite{PP_ct} in this case.  The main result in~\cite{Ben}
(and in greater generality in~\cite{Bar}), is that
these upper bounds are asymptotically sharp.

\smallskip
\subsection{} \label{ss:finrem-radom}
For the uniform random partitions \ts $\la\vdash n$, we have \.
$\ell(\la) = \Theta(\sqrt{n}\.\log n)$ \. and \. $d(\la) = \Theta(\sqrt{n})$,
see e.g.\ \cite{DP} and references therein.  This implies that the
bounds in \eqref{eq:cor-rows} and~\eqref{eq:main-upper} are useful
only for partitions with relatively few rows and small Durfee square size,
respectfully.

\smallskip

\subsection{} \label{ss:finrem-const}
Define
$$
\rB(n,k) \, := \, \max \. \big\{ \ts g(\la,\mu,\nu) \ : \ \la,\mu,\nu\vdash n \ \
\text{and} \ \ d(\la),d(\mu),d(\nu)\le k \ts \big\}.
$$
Comparing the bounds in Theorem~\ref{t:intro-rows-lower} and Theorem~\ref{t:main},
it would be natural to believe that the upper bound on \ts $\rB(n,k)$ \ts in
\eqref{eq:main-upper} is closer to the truth than the lower bound in~\eqref{eq:main-lower-A}.

\begin{conj}\label{conj:const-4}
There is a universal constant \ts $c>0$ \ts such that
$$
\rB(n,k) \, \ge \, n^{4k^3 \ts - \ts ck^2} \quad\text{for all} \ \ n, \ts k \ge\ts 1\ts.
$$
\end{conj}

\smallskip

\subsection{} \label{ss:finrem-rho}
We believe that the Kronecker coefficients in Theorem~\ref{t:explicit} grow much
faster than our lower bounds suggest.
The following conjecture  immediately implies Conjecture~\ref{conj:rKfs}
improving upon Theorem~\ref{t:Kron-fs-asy}.

\begin{conj}  \label{conj:rho}
We have:
$$\aligned
g(\rho_k,\rho_k,\rho_k) \, & = \, \sqrt{n!} \, e^{-O(n)} \quad \text{where}
\quad n=\tbinom{k}{2}, \qquad\text{and} \\
g(\de_\ell,\ts\de_\ell,\ts\de_\ell) \, & = \, \sqrt{n!} \, e^{-O(n)} \quad \text{where}
\quad n=\ell^2.
\endaligned
$$
\end{conj}

\smallskip

\subsection{} \label{ss:finrem-BBS}
Let \ts $n=\binom{k}{2}$. Define
$$
\rF(n) \, := \,  \max \. \big\{ \ts g(\rho_k\ts,\ts\rho_k\ts,\la) \ : \ \la\vdash n \ \
\text{and} \ \ \la=\la' \ts \big\}.
$$
It would be interesting to find sharp lower bounds on \ts $\rF(n)$.

It was shown in \cite[$\S4.2$]{PP2}, that \ts $g(\rho_k\ts,\ts\rho_k\ts,\la) = e^{\Om(\sqrt{n})}$ \ts
for two-row partitions \ts $\la=(n/2,n/2)$, where~$n$ is even.  For self-conjugate \ts $\la$,
it was only shown recently in~\cite[$\S$5]{BBS} using \defng{modular representation theory},
that the \ts $\rF(n)$ \ts is unbounded.

Combined with the lower bound for Littlewood--Richardson coefficients given in \cite[Thm~1.5]{PPY},
Theorem~5.11 in~\cite{BBS} implies
that \. $\rF(n) = e^{\Om(\sqrt{n})}$.  This is nowhere close to Conjecture~\ref{conj:rho}, but
gives us a hope that there might be more tools to be discovered.

\smallskip

\subsection{} \label{ss:finrems-MPP}

In~\cite{MPP}, there is a tight  asymptotic bound \,
$g\big(\delta_{2s},\delta_{2s},(n-k,k)\big) \ts = \ts \Theta\big(2^{\sqrt{2k}}/k^{3/2} \big)$ \. in the case
when \. $k/n \in (0,1/2)$. However, this bound cannot be applied when \ts $k=n/2+ o(n)$,
so the bound  from~\cite{PP2} is still the best known lower bound in this case.

\vskip.6cm

\subsection*{Acknowledgements}
We are grateful to Sasha Barvinok,
Chris Bowman, Christian Ikenmeyer, Rosa Orellana,
Fedya Petrov and Mike Zabrocki for helpful discussions and remarks
on the subject. We also thank Stefan Trandafir for reminding us
of the polynomial growth problem for fixed Durfee size. We thank the anonymous referee for many helpful remarks leading to improved exposition.  Both authors were partially supported by the~NSF.

\vskip.7cm

{\em \,
We dedicate this paper to the memory of Christine Bessenrodt who passed away
earlier

this year.  May her memory be a blessing and may her life
be an inspiration for all of us.}

\vskip1.6cm

%\newpage

\end{document}